\include{birkart.sty}

\documentclass{amsart}[12pt]
\usepackage{amsmath}
\usepackage{amssymb}
\usepackage{amsfonts}
\usepackage{graphicx}
\usepackage{texdraw}
\usepackage{tikz}
\usepackage{enumitem}
\usepackage{graphpap}
\usepackage{wasysym}
\usepackage{color}
\usepackage{pxfonts}
\usepackage[OT2,T1]{fontenc}

\input txdtools

\newtheorem{thm}{Theorem}[section]
\newtheorem{cor}[thm]{Corollary}
\newtheorem{lem}[thm]{Lemma}

\theoremstyle{definition}

\theoremstyle{remark}
\newtheorem{rem}{Remark}

\numberwithin{equation}{section}

\newcommand{\Prob}{{\mathbb{P}}}
\renewcommand{\d}{\partial}
\newcommand{\dbar}{\bar{\partial}}
\newcommand{\calA}{{\mathcal A}}
\newcommand{\calB}{{\mathcal B}}

\newcommand{\calN}{{\mathcal N}}

\newcommand{\calE}{{\mathcal E}}

\newcommand{\calW}{{\mathcal W}}

\newcommand{\bigO}{{\mathcal O}}

\newcommand{\T}{{\mathbb T}}

\newcommand{\1}{{\mathbf{1}}}

\newcommand{\ttn}{{\tt{n}}}
\newcommand{\tts}{{\tt{s}}}

\renewcommand{\Cap}{\operatorname{Cap}}

\newcommand{\R}{{\mathbb R}}
\newcommand{\Z}{{\mathbb Z}}
\newcommand{\E}{{\mathbb E}}

\newcommand{\C}{{\mathbb C}}
\newcommand{\D}{{\mathbb D}}

\newcommand{\Pol}{\operatorname{Pol}}
\newcommand{\supp}{\operatorname{supp}}
\newcommand{\re}{\operatorname{Re}}

\newcommand{\erfc}{\operatorname{erfc}}
\newcommand{\Int}{\operatorname{Int}}
\newcommand{\Ext}{\operatorname{Ext}}

\newcommand{\dist}{\operatorname{dist}}

\newcommand{\He}{\operatorname{He}}

\makeatletter
\newcommand*\bigcdot{\mathpalette\bigcdot@{.5}}
\newcommand*\bigcdot@[2]{\mathbin{\vcenter{\hbox{\scalebox{#2}{$\m@th#1\bullet$}}}}}
\makeatother

\allowdisplaybreaks

\begin{document}



\title[Correlations for two-dimensional outpost ensembles]{Szeg\H{o} type correlations for two-dimensional outpost ensembles}

\begin{abstract} We consider two-dimensional Coulomb systems for which the coincidence set contains an outpost in the form of a suitable Jordan curve. We study asymptotics for correlations along the union of the outpost and the outer boundary of the droplet. These correlations turn out to have a universal character and are given in terms of the reproducing kernel for a certain Hilbert space of analytic functions, generalizing the Szeg\H{o} type edge correlations obtained recently by Ameur and Cronvall. There are several additional results, for example on the effect of insertion of an exterior point charge in the presence of an outpost.
\end{abstract}

\author{Yacin Ameur}
\address{Yacin Ameur\\
Department of Mathematics\\
Lund University\\
22100 Lund, Sweden}
\email{ Yacin.Ameur@math.lu.se}

\author{Ena Jahic}
\address{Ena Jahic\\
Department of Mathematics\\
Lund University\\
22100 Lund, Sweden}
\email{ena.jahic@math.lth.se}

\maketitle

\section{Introduction and main results}
We continue the investigation of two-dimensional outpost ensembles, such as introduced recently in \cite{AC,ACC}. Thus we study Coulomb systems which accumulate in the vicinity of a connected droplet ($S$), but with a number of outliers dispersed along a smooth Jordan curve ($C_2$) in the exterior of the droplet. See Figure \ref{fig1}.

Under suitable conditions the number of outliers (near $C_2$) has an asymptotic Heine distribution; in particular the expected number is strictly positive and finite even in the thermodynamic limit, as the total number of particles increases indefinitely (cf. \cite{AC,ACC}). The statistical behaviour differs starkly from outpost ensembles appearing in the literature on Hermitian random matrices, \cite{BL,Cl,Mo}.

The outpost regime is critical in the following sense. Under Laplacian growth it represents a point at which the droplet changes topology, the outpost being the ``germ'' of a new  ring-shaped component. Drawing on the principle that criticality merits special attention, we shall delve deeper into this particular regime and shed some new light on the associated function theory.

In the spirit of \cite{AC1,ACC,FJ}, we shall study asymptotic long-range correlations along the union of the outpost and the outer edge of the droplet. We will find that these correlations are conveniently expressed in terms of the reproducing kernel of a suitable Hilbert space of analytic functions, providing a natural generalization of the Szeg\H{o} type asymptotics in \cite{AC1}.

We remark that several aspects of the ``post-critical'' regime of disconnected droplets are investigated in \cite{AC,ACC,ACC1}. It is shown in \cite{ACC1} that, assuming radial symmetry, the Szeg\H{o} kernel describing correlations along the edge of a spectral gap oscillates with the number of particles. Intuitively, oscillations in the correlation functions are caused by random \emph{bilateral} displacements of particles back and forth across a spectral gap. In the present case the displacements are unilateral, from $S$ to $C_2$, and this is not enough to cause oscillations.

\subsection{Two-dimensional outpost ensembles} We recall the outpost model from \cite{AC}. While the following conditions may seem restrictive at first glance, we reassure the reader that examples are abundant. See Section \ref{mex} for several related comments and specific examples, cf.~also \cite{AC}.

\subsubsection{Background from potential theory}
We consider Coulomb systems $\{z_j\}_1^n\subset \C$ picked randomly with respect to the Gibbs measure
\begin{align}\label{bogi}d\Prob_n(z_1,\ldots,z_n)=\frac 1 {Z_n}\prod_{1\le i<j\le n}|z_i-z_j|^2\prod_{i=1}^n e^{-nQ(z_i)}\, dA(z_i)\end{align}
where $Q:\C\to\R\cup\{+\infty\}$ is 
a subharmonic function with values in $\R\cap\{+\infty\}$ which is finite on some open set and satisfies
$$\liminf_{|z|\to\infty}\frac {Q(z)}{2\log|z|}>1.$$

Here and throughout we write
$dA(z)=\pi^{-1}d^2 z$ for the normalized area measure on $\C$.

For convenient reference, we now recall some notions and results from weighted potential theory. We refer to \cite{ST} for proofs and further details.

\begin{center}
\begin{figure}[t]
\includegraphics[width=0.6\textwidth]{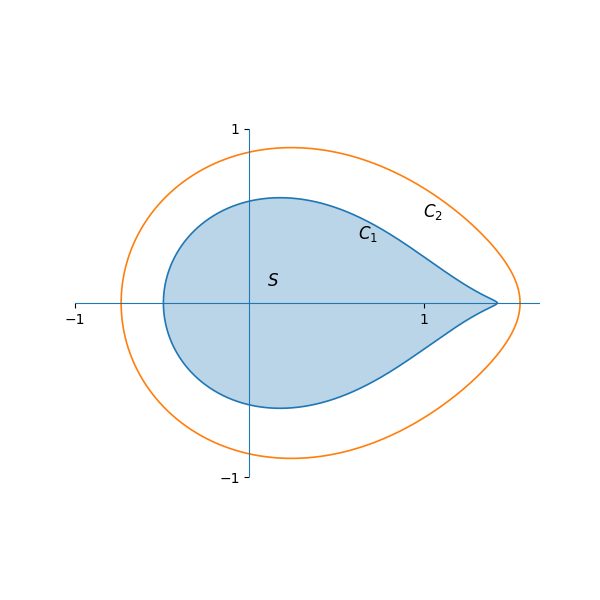}
\caption{The coincidence set $S^*=S\cup C_2$ for the outpost potential described in Section \ref{mex} with $\alpha=-0.5$, $r_1=0.735$ and $r_2=1$.
}
\label{fig1}
\end{figure}
\end{center}

The Frostman equilibrium measure with respect to $Q$ is the unique compactly supported probability measure $\sigma=\sigma[Q]$ on $\C$ which minimizes the weighted energy functional
$$I_Q[\mu]=\iint_{\C^2}\log \frac 1 {|z-w|}\,d\mu(z)d\mu(w)+\int_\C Q\, d\mu$$
over all compactly supported Borel probability measures $\mu$ on $\C$. The support $S=S[Q]=\supp\sigma$ is called the droplet with respect to $Q$. 

Assuming (as we shall) that $Q$ is $C^2$-smooth in a neighbourhood of $S$, then $\sigma$ is absolutely continuous with respect to the normalized area measure $dA(z)=\frac 1 \pi d^2 z$, and in fact (cf.~\cite{ST})
\begin{align}\label{bo}d\sigma(z)=\Delta Q(z)\cdot\1_S(z)\, dA(z).\end{align}

Here and throughout, $\Delta:=\d\dbar=\frac 1 4 (\d_x^2+\d_y^2)$ is the normalized Laplacian, i.e., the standard Laplacian divided by $4$.

Next we define the obstacle function $\check{Q}(z)$ to be the pointwise supremum of $s(z)$, where $s(w)$ ranges over all subharmonic functions $s:\C\to \R$ which satisfy $s(w)\le 2\log|w|+\bigO(1)$ as $w\to\infty$. Then $\check{Q}$ is globally $C^{1,1}$-smooth and $\check{Q}$ is harmonic on $\C\setminus S$, see \cite{ST}. In particular $d\sigma=\Delta \check{Q}\, dA$.

We define the \emph{coincidence set} for the obstacle problem by
$$S^*:=\{Q=\check{Q}\}.$$ 

This set satisfies $S\subset S^*$ (again cf. \cite{ST}). 

While the overwhelming majority of investigations on Coulomb gas ensembles assume that $S=S^*$, we shall here, following \cite{ACC,AC}, investigate an interesting class of potentials for which the ``shallow set'' $S^*\setminus S$ has positive capacity.

\subsubsection{Class of outpost potentials} Given the droplet $S$, we will write $U$ for the connected component of $\hat{\C}\setminus S$ containing infinity. 

We shall assume that the boundary $\d U$ is a single, regular Jordan curve $C_1$. We shall also assume that $(S^*\setminus S)\cap U$ is a Jordan curve $C_2$ and that $Q(z)$ is real-analytic and strictly subharmonic in a neighbourhood of $C_1\cup C_2$.

Let $\D_*=\{|z|>1\}\cup\infty$ and let $\phi_1:U\to\D_*$ be the conformal mapping of the form
\begin{equation}\phi_1(z)=\frac 1 {r_1}z+a_{0,1}+a_{1,1}\frac 1 z+\cdots\end{equation}
for $z$ in a neighbourhood of $\infty$, where $r_1=\Cap C_1>0$. (See Figure \ref{fig2}.)

Following \cite{AC}, we impose two assumptions, or \textit{compatibility conditions}, restricting the class of curves $C_2$ and the values of the Laplacian $\Delta Q$ along $C_1\cup C_2$.

The first compatibility condition is that $\phi_1$ should map the outpost $C_2$ to a circle centered at the origin, of radius $r_2/r_1$ where $r_2>r_1$ is arbitrary but fixed. The normalized conformal map $\phi_2:\Ext C_2\to\D_*$ is then simply
$\phi_2=\frac {r_1}{r_2}\phi_1$ and $r_2=\Cap C_2$.

\begin{center}
\begin{figure}
\includegraphics[width=\linewidth]{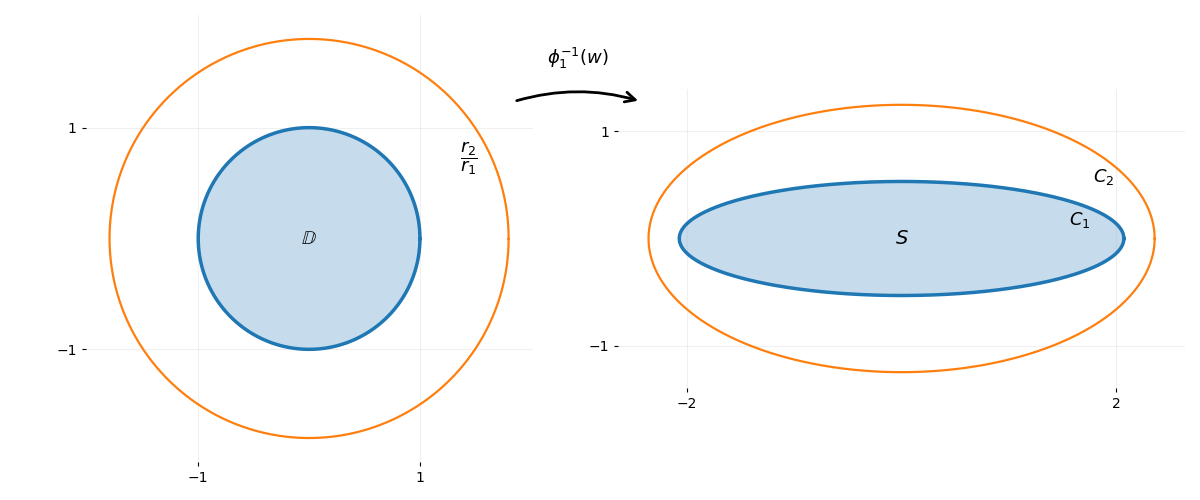}
\caption{Outpost potential modeled on an elliptic Ginibre potential, cf.~Section \ref{mex}} 
\label{fig2}
\end{figure}
\end{center}

Next consider the ring-shaped domain $G=(\Ext C_1)\cap (\Int C_2)$. Let
$$\varpi(z)=\frac {\log|\phi_1(z)|}{\log(r_2/r_1)}$$
be the harmonic function which is $0$ along $C_1$ and $1$ along $C_2$. Also let $H(z)$ be the harmonic real-valued function on $G$ with boundary values $H(z)=\frac 1 2 \log \Delta Q(z)$ on $C_1\cup C_2$.

Standard results about the Dirichlet problem (e.g.~\cite{Bell}) imply that there exists a holomorphic function $h_1$ on $G$ and a real constant $c$ such that
\begin{equation}\label{constc}H(z)=\re h_1(z)+c\varpi(z),\qquad (z\in G).\end{equation}

The  
second compatibility condition is that $h_1(z)$ should extend to a holomorphic function on $U$. We fix this function uniquely by requiring that $h_1(\infty)$ is a real number.

\subsubsection{The correlation kernel} Consider the point process $\{z_j\}_1^n$ picked randomly with respect to \eqref{bogi}.

For $k\le n$, the $k$-point correlation function $R_{n,k}$ is defined by the requirement that, for all bounded continuous functions $f$ on $\C^k$,
$$\E_n(f(z_1,\ldots,z_k))=\frac {(n-k)!}{n!}\int_{\C^k}f\cdot R_{n,k}\, dA^{\otimes k}.$$

The process $\{z_j\}_1^n$ is determinantal, i.e., there exists a kernel $K_n(z,w)$ such that $R_{n,k}(w_1,\ldots,w_k)=\det(K_n(w_i.w_j))_{i,j=1}^k$.

Indeed, a correlation kernel $K_n(z,w)$ for the process $\{z_j\}_1^n$ can be constructed in the following way.

Consider the space of weighted polynomials
$$\calW_n=\big\{f(z)=p(z)\cdot e^{-\frac 1 2 nQ(z)}\,;\, p\in \Pol(n-1)\big\}$$
where $\Pol(n-1)$ is the set of holomorphic polynomials of degree at most $n-1$. We equip $\calW_n$ with the norm inherited from $L^2(\C;dA)$.

The reproducing kernel of the space $\calW_n$ is then a correlation kernel for $\{z_j\}_1^n$; in the following this canonical correlation kernel is always chosen. We can write
\begin{align}\label{ck}K_n(z,w)=\sum_{j=0}^{n-1}e_{j,n}(z)\overline{e_{j,n}(w)}\end{align}
where the wavefunction
\begin{align}\label{wavefunction}e_{j,n}(z)=\gamma_{j,n}p_{j,n}(z)e^{-\frac 1 2 nQ(z)}\end{align}
is normalized by $p_{j,n}(z)=z^j+a_{j-1,n}z^{j-1}+\cdots +a_{0,n}$ the $j$:th monic orthogonal polynomial with respect to the norm
$$\|p\|_{nQ}^2:=\int_\C |p|^2e^{-nQ}\, dA$$
and $\gamma_{j,n}=1/\|p\|_{nQ}$.

By a \textit{cocycle} we mean a function $c(z,w)$ of the form $c(z,w)=a(z)\overline{a(w)}$, where $a(z)$ is a continuous unimodular function. Then, for any kernel $K(z,w)$ we have $\det(K(z_i,z_j)c(z_i,z_j))_{i,j=1}^n=\det(K(z_i,z_j))_{i,j=1}^n$. I.e., the kernels $K$ and $cK$ give rise to the same determinantal process.

Given these proviso, we consider the
reproducing kernel $K_n(z,w)$ for the space $\calW_n$. We shall derive asymptotics as $n\to\infty$ for $c_n(z,w)K_n(z,w)$ in cases when $z,w$ are in a vicinity of $C_1\cup C_2$, where $c_n(z,w)$ are suitable cocycles. These correlations are expressed in terms of the reproducing kernel of a Hilbert space of analytic functions on $U$, which we now turn to.

\subsection{Associated Szeg\H{o} kernels and the Heine distribution}
For a function $f(z)$ analytic in $U$, having appropriate boundary values along $C_1$, we define a squared norm by
\begin{align*}\label{ha1}\|f\|_1^2=\oint_{C_1}|f(z)|^2\frac {|dz|}{\sqrt{\Delta Q(z)}}.
\end{align*} 
(Here and throughout, $|dz|$ denotes the usual arclength measure.)

As above, let $h_1(z)$ be the analytic function in the exterior $\Ext C_1$ (the connected component in $\hat{\C}$ containing $\infty$) which is real at $\infty$ and solves the boundary value problem
$$\re h_1=\frac 1 2\log \Delta Q,\qquad (\text{along}\quad C_1).$$

Let $H_1$ be the Hilbert space of all analytic functions $f:U\to\C$ such that $\|f\|_1<\infty$ and $f(\infty)=0$.
This is a weighted Hardy space of the type considered in \cite{AC1,ACC1}. We denote by
\begin{equation}\label{s1def}S_1(z,w):=\frac 1 {2\pi}\frac {\sqrt{\phi_1'(z)}\overline{\sqrt{\phi_1'(w)}}e^{\frac 1 2(h_1(z)+\overline{h_1(w)})}}{1-\phi_1(z)\overline{\phi_1(w)}}\end{equation}
the reproducing kernel for the space $H_1$. (Here and throughout, $\sqrt{\phi_1'(z)}$ denotes the branch of the square root which is real at infinity.)

In addition, we shall consider the following equivalent norm on $H_1$
\begin{align}\label{sqn}\|f\|_{1,2}^2:=\oint_{C_1}|f(z)|^2\frac {|dz|}{\sqrt{\Delta Q(z)}}+\oint_{C_2}|f(z)|^2\frac {|dz|}{\sqrt{\Delta Q(z)}}.\end{align}

We denote by $H_{1,2}$ the Hilbert space of analytic functions on $U$ normed by $\|\cdot\|_{1,2}$ and by
$S_{1,2}(z,w)$ the reproducing kernel for $H_{1,2}$; as we shall see below this kernel has several useful representations. 

\begin{lem} \label{boll} When $z,w\in\Ext C_1$ we have the representation
\begin{align}\label{grund}S_{1,2}(z,w)=\sqrt{\phi_1'(z)}\overline{\sqrt{\phi_1'(w)}}e^{\frac 1 2(h_1(z)+\overline{h_1(w)})}\frac 1 {2\pi}
\sum_{j=1}^\infty\frac 1 {(\phi_1(z)\overline{\phi_1(w)})^j}\frac {r_1^{1-2j}}{r_1^{1-2j}+e^{-c}r_2^{1-2j}}.
\end{align}
\end{lem}

A proof is given in Section \ref{sec1}. 

\smallskip

It is important to note that $S_{1,2}(z,w)$ continues meromorphically, as a function of $(z,\bar{w})$, to a larger domain. 

Namely, let $D$ be any domain in $\hat{\C}$ containing $C_1\cup \Ext C_1$ such that $\phi_1$ continues analytically to $D$ and satisfies $|\phi_1|>r_1/r_2$ there. We also assume that $h_1(z)$ and $q_1(z)$ have analytic continuations to $D$.

It follows easily from Lemma \ref{boll} that (cf.~Section \ref{sec1} for details)
\begin{align}\label{extend}S_{1,2}&(z,w)=S_1(z,w)\\
&+\sqrt{\phi_1'(z)}\overline{\sqrt{\phi_1'(w)}}e^{\frac 1 2(h_1(z)+\overline{h_1(w)})}\frac 1 {2\pi}
\sum_{j=1}^\infty\frac 1 {(\phi_1(z)\overline{\phi_1(w)})^j}\frac {e^{-c}r_2^{1-2j}}{r_1^{1-2j}+e^{-c}r_2^{1-2j}},\nonumber
\end{align}
where $S_1(z,w)$ is the Szeg\H{o} kernel in \eqref{s1def}. Note that the sum in the second line of \eqref{extend} is analytic in
$(z,\bar{w})$ for all $(z,w)\in D\times D$, so the formula provides a sesqui-meromorphic continuation.

We proceed to define some further objects required in order to formulate our main results.

\smallskip

Let $q_1(z)$ be the analytic function in $\Ext C_1$ solving the boundary value problem $\re q_1= Q$ on $C_1$ (and having $q_1(\infty)$ real). Likewise, let $q_2(z)$ be the analytic function on $\Ext C_1$ such that $\re q_2= Q$ along $C_2$. 

In \cite[Lemma 2.1]{AC} it is shown that $\phi_1(z)e^{\frac 1 2 q_1(z)}=\phi_2(z)e^{\frac 1 2 q_2(z)}$ on $\Ext C_2$. We can thus unambiguously define an analytic function by
\begin{align}\label{udef}u(z):=\phi_1(z)e^{\frac 1 2 q_1(z)}=\phi_2(z)e^{\frac 1 2 q_2(z)}
\end{align}
for $z$ in the domain $D$ above. Closely related is the following identity for the obstacle function
\begin{equation}\label{fundid}\check{Q}(z)=2\log|\phi_1(z)|+\re q_1(z),\qquad(z\in \Ext C_1).\end{equation}

\smallskip

We next recall a few facts concerning the Heine distribution.

A discrete random variable $X$, taking values in $\Z_+=\{0,1,2,\ldots\}$, is said to have a Heine distribution with parameters $\theta>0$ and $0<q<1$ if 
\begin{align}\Prob(\{X=k\})=\frac 1 {(-\theta;q)_\infty}\frac {q^{\frac 1 2 k(k-1)}\theta^k}{(q;q)_k},\qquad k\ge 0,
\end{align}
where the $q$-Pochhammer symbols are defined by
$$(z;q)_k=\prod_{j=0}^{k-1}(1-zq^j),\qquad (z;q)_\infty=\prod_{j=0}^\infty (1-zq^j).$$

The fact that this defines a probability distribution on $\Z_+$ follows by the $q$-binomial theorem, see \cite{ACC} and references therein; we denote this distribution by $\He(\theta,q)$.

For convenient reference, we recall the following result from \cite[Theorem 1.2]{AC} (cf.~\cite[Corollary 1.10]{ACC}).

\begin{thm}[\cite{AC}, Theorem 1.2]\label{compa} Let $\calN$ be a fixed neighbourhood of $C_2$, small enough that its closure does not intersect the droplet $S$. Let $\{z_j\}_1^n$ be a random sample from \eqref{bogi}, and define $X_n$ to be the number of $j$'s between $1$ and $n$ such that $z_j\in \calN$. Then $X_n$ converges in distribution, as $n\to\infty$, to a Heine distribution $X\sim\He(\theta,q)$ with parameters
\begin{equation}\theta=\frac {r_1}{r_2}e^{-c},\qquad q=\bigg(\frac {r_1}{r_2}\bigg)^2.\end{equation}
The expectation of this Heine distribution is given by
\begin{equation}\label{expheine}\mu:=\E X=\sum_{j=1}^\infty\frac {r_2^{1-2j}}{r_1^{1-2j}e^c+r_2^{1-2j}}.\end{equation}
\end{thm}

We will not use Theorem \ref{compa} in what follows; in fact, we shall give an alternative proof which implies a slightly better convergence.

\subsection{Main results on kernel asymptotics}

For $k=1,2$ we denote the ``belt'' about $C_k$ by
\begin{equation}\label{2cho}\calB_k:=\bigg\{z\,;\, \dist(z,C_k)\le M\sqrt{\frac{\log n}n}\bigg\}\end{equation}
where $M$ is a fixed, large enough constant.

We have the following result.

\begin{thm} \label{mth0} Suppose that 
$(z,w)$ satisfies
\begin{align}\label{uni1}z,w\in\calB_1\cup \Ext C_1\qquad \text{and}\qquad |\phi_1(z)\overline{\phi_1(w)}-1|\ge \eta,\end{align}
where $\eta>0$ is arbitrary but fixed. Also let $\beta>0$ be an arbitrary but fixed constant in the range $0<\beta<\frac 1 4$.
We then have the convergence (with $u(z)$ given by \eqref{udef})
\begin{align}\label{conv1}K_n(z,w)&=\sqrt{2\pi n}
\cdot (u(z)\overline{u(w)})^n e^{-\frac n 2(Q(z)+Q(w))}\cdot S_{1,2}(z,w)\cdot (1+\bigO(n^{-\beta})),\qquad n\to\infty,
\end{align}
where the implied constant is uniform in $(z,w)$ satisfying \eqref{uni1}. 
\end{thm}

We note that Theorem \ref{mth0} generalizes the Szeg\H{o} type convergence obtained by Ameur and Cronvall for connected coincidence sets in \cite[Theorem 1.3]{AC1}. 

\smallskip

Combining the above result with \eqref{fundid}, we obtain the following consequence.

\begin{cor} Let $V(z)$ be the harmonic continuation of $\check{Q}|_{\Ext C_1}$ inwards across $C_1$ to a neighbourhood $D$ of $C_1\cup \Ext C_1$. Then for all $(z,w)$ satisfying \eqref{uni1}
\begin{align}\label{fest}|K_n(z,w)|&=\sqrt{2\pi n}\, e^{-\frac n 2(Q-V)(z)-\frac  n 2(Q-V)(w)}|S_{1,2}(z,w)|\cdot (1+\bigO(n^{-\beta})).\end{align}
\end{cor}

The estimate \eqref{fest} implies that that $K_n(z,w)$ is negligible when $z$ or $w$ is in $\Ext C_1\setminus (\calB_1\cup\calB_2)$. Indeed, there is a constant $c>0$ such that
\begin{equation}\label{bock}(Q-\check{Q})(z)\ge c(\dist(z, C_1\cup C_2))^2,\qquad z\in\Ext C_1.\end{equation}
This latter estimate follows by a straightforward Taylor expansion, as in the proof of \cite[Lemma 2.1]{A1}. Using \eqref{fest}, \eqref{bock} we obtain the following corollary.

\begin{cor} \label{qm} Suppose that $z,w\in \calB_1\cup\Ext C_1$ and at least one of the points $z,w$ are in the complement $\Ext C_1\setminus(\calB_1\cup\calB_2)$. Given an arbitrarily large constant $N$, we may choose $M$ in \eqref{2cho} large enough that
$|K_n(z,w)|\le Cn^{-N}$ where $C$ is a constant depending only on $Q$, $M$ and $N$.
\end{cor}

We next turn to detailed asymptotics in cases when both $z,w$ are in $\calB_1\cup\calB_2$.

\smallskip

 If $p\in C_1\cup C_2$ we write $\nu(p)$ for the exterior unit normal to the curve $C_j$ to which $p$ belongs.

Given two points $p,q\in C_1\cup C_2$ we consider nearby points $z,w$ of the form
\begin{align}\label{zoom}z=p+\frac s {\sqrt{2n\Delta Q(p)}}\nu(p),\qquad w=q+\frac t {\sqrt{2n\Delta Q(q)}}\nu(q),
\end{align}
where $s,t$ are real numbers with $|s|,|t|\le M\sqrt{\log n}$.

We have the following theorem for the case when at least one of the points $z,w$ is near $C_2$. (Here and in what follows, $\beta$ denotes an arbitrary but fixed number in the interval $0<\beta<\frac 1 4$.)

\begin{thm}\label{mth1} Suppose that $p,q\in C_1\cup C_2$. If both $p$ and $q$ are in $C_1$ we also assume that $|p-q|\ge \eta$ for some sufficiently small $\eta>0$.  Then there are cocycles $c_n(z,w)$ such that, as $n\to\infty$,
\begin{align}c_n(z,w)K_n(z,w)=\sqrt{2\pi n}\cdot S_{1,2}(p,q)\cdot e^{-\frac 1 2(s^2+t^2)}\cdot (1+\bigO(n^{-\beta})).
\end{align}
The implied $\bigO$-constant is  uniform for all such $p$ and $q$ in question. 
\end{thm}

We also have the following result for the density $K_n(z,z)$ near the outpost, proving Gaussian decay as one moves away from $C_2$.

\begin{thm}\label{mcor1} Let $p\in C_2$ and set 
\begin{align}\label{scal}z=p+\frac t {\sqrt{2n\Delta Q(p)}}\nu(p).\end{align}
Then, as $n\to\infty$, we have uniformly for $|t|\le M\sqrt{\log n}$,
\begin{align*}K_n(z,z)&=\sqrt{\frac{n\Delta Q(p)}{2\pi}}\cdot \mu\cdot |\phi_2'(p)|\cdot e^{-t^2}\cdot (1+\bigO(n^{-\beta})),
\end{align*}
where $\mu$ is the expectation of the Heine distribution given in formula \eqref{expheine}.
\end{thm}

The behaviour of $K_n(z,z)$ for $z$ near $C_1$ is more involved and will not be studied here. See however Section \ref{comrel} for several related comments.

\smallskip

We note the following consequence of Theorem \ref{mcor1}. 

Let $\{z_j\}_1^n$ be a random sample and let $Y_n$ denote the number of indices $j$, $1\le j\le n$, such that $z_j\in\calB_2$, where $\calB_2$ is the neighbourhood of $C_2$ in \eqref{2cho}. Evidently
$$\E_n [Y_n]=\int_{\calB_2}K_n(z,z)\, dA(z).$$

We can now state the following result, which improves on the rate of convergence
in Theorem \ref{compa}.

\begin{cor} \label{chuck}  As $n\to\infty$, we have the convergence $\E_n[Y_n]=\mu\cdot (1+\bigO(n^{-\beta})).$
\end{cor}

\subsection{Berezin measures and analytic carriers} Fix a point $z\in\Ext C_1$ and consider the following probability measure on $\C$
\begin{align}d\mu_{n,z}(w)=\frac {|K_n(z,w)|^2}{K_n(z,z)}\, dA(w).\end{align}
This is known as a Berezin measure; probabilistically it measures the repulsive effect obtained by inserting a point charge at $z$, cf.~\cite[Section 7.6]{AHM}.

Asymptotic properties, as $n\to\infty$, of Berezin measures have been well studied in cases where there are no outposts. For example, if $z$ is in the droplet, then under mild assumptions, $\mu_{n,z}$ converges weakly to the Dirac measure $\delta_z$, while if $z$ is in the exterior, then $\mu_{n,z}$ converges to the harmonic measure $\omega_z$ of the exterior domain, see e.g.~\cite[Section 7]{AHM}.

We turn to the asymptotics of the measures $\mu_{n,z}$ in the present setting with an outpost $C_2$. In the following, the point $z$ is fixed in $\Ext C_1$, and we identify $\mu_{n,z}$ with the functional
$$\mu_{n,z}(f):=\int_\C f\, d\mu_{n,z}.$$

Let us denote by $\calA(\Ext C_1)$ the algebra of analytic functions on $\Ext C_1$ which extend continuously to $C_1$.

We have the following theorem.

\begin{thm} \label{massthm} Let $f\in \calA(\Ext C_1)$. Extend $f$ by continuity to $\hat{\C}$ in some way. Then, for any given $N>0$, we have as $n\to\infty$
$$\mu_{n,z}(f)=(b_z^{(1)}(f)+b_z^{(2)}(f))\cdot (1+\bigO(n^{-\beta}))+\|f\|_\infty\cdot \bigO(n^{-N})$$
where $b_z^{(1)}$ and $b_z^{(2)}$ are the functionals
\begin{align}b_z^{(k)}(f):=\oint_{C_k}f(q)\frac {|S_{1,2}(z.q)|^2}{S_{1,2}(z,z)}\,|dq|,\qquad (k=1,2);\end{align}
$\|f\|_\infty$ denotes the sup-norm of the extended function.

Moreover, we have the following reproducing property
\begin{equation}\label{repro0}b_z^{(1)}(f)+b_z^{(2)}(f)=f(z).\end{equation}
Finally, if we set $r=|\phi_2(z)|$ (so $r>r_1/r_2$), then the total mass of the measure $b_z^{(2)}$ is
$$b_z^{(2)}(1)=\frac {\sum_{j=1}^\infty r^{-2j}\big(\frac {r_2^{1-2j}}{r_1^{1-2j}e^{-c}+r_2^{1-2j}}\big)^{\,2}}{\sum_{j=1}^\infty r^{-2j}\frac {r_2^{1-2j}}{r_1^{1-2j}e^{-c}+r_2^{1-2j}}}.$$
\end{thm}

The equation \eqref{repro0} says that the functional $b_z^{(1)}+b_z^{(2)}$ has the same exterior analytic carrier as the point evaluation $\delta_z$, over functions $f\in \calA(\Ext C_1)$, i.e., $b_z^{(1)}(f)+b_z^{(2)}(f)=\delta_z(f)$.

In the well studied case without outposts, the Berezin measure rooted at an exterior point $z$ is supported on $C_1$ and is given by
$$d\omega_z(q):=\frac {|S_1(z,q)|^2}{S_1(z,z)}\,|dq|,\qquad (q\in C_1),$$
where $S_1(z,w)$ is the Szeg\H{o} kernel \eqref{s1def}. Moreover, $\omega_z$ is the harmonic measure for $\Ext C_1$.
(See \cite{AC1,AHM,GaM} and references therein.) 

In conclusion, the three probability measures $\delta_z$, $\omega_z$, $b_z^{(1)}+b_z^{(2)}$ all have the same carrier with respect to $\calA(\Ext C_1)$.

\begin{rem}The total mass $b_z^{(2)}(1)$ of the measure $b_z^{(2)}$ is an increasing function of $r=\phi_2(z)$; it increases from $0$ at $r=\frac {r_1}{r_2}$ to
$\frac {r_2^{-1}}{r_1^{-1}e^{-c}+r_2^{-1}}$, as $r\to+\infty$, see figure \ref{fig4}.
\end{rem}

\begin{center}
\begin{figure}
\includegraphics[width=0.6\textwidth]{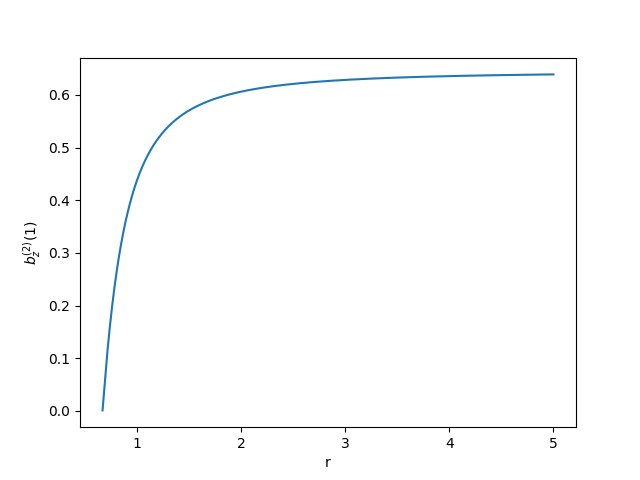}
\caption{The total mass of the measure $b^{(2)}_z$ as a function of $r=\phi_2(z)$, $r\ge \frac {r_2}{r_1}$, with $r_1=1, \, r_2=1.5, \, c=1$.}
\label{fig4}
\end{figure}

\end{center}

\subsection{A few concrete examples} \label{mex} As noted in \cite{AC}, it is easy to generate examples of outpost potentials meeting all conditions imposed above. For example, let us explain how the ensemble in Figure \ref{fig1} was generated. The point of departure is a class of 1-point unbounded quadrature domains considered in the recent work \cite{GrM}. (The complement $\hat{\C}\setminus S$ is a quadrature domain in a suitable generalized sense, see \cite{AC1,GrM} and references therein.)

As in \cite[Section 1.3.1]{AC} we could start with any Hele-Shaw potential, i.e., a potential whose Laplacian is constant in a neighbourhood of the droplet. Let us pick a potential $Q_1(z)$ of the form studied in \cite[Section 2]{GrM} 
\begin{equation}Q_1(z)=\begin{cases}\frac 1 {r_1}\bigg(|z|^2-2\re(\alpha\cdot \log (z-2))\bigg)& z\in K,\cr
+\infty & \mathrm{otherwise},\end{cases}\label{belem}\end{equation}
where
$K$ is a simply connected compact subset of $\C$, chosen so that the droplet $S=S[Q_1]$ is compactly contained in the interior of $K$ while $2\not\in K$; $\alpha$ is a complex parameter such that $|\alpha|<2+\re\alpha$. The droplet $S$ then has normalized area $\int_S dA=\sqrt{r_1}$. Write $C_1=\d S$. 

For simplicity, we choose $-1<\alpha<0$.
As is shown in \cite[Section 2.1]{GrM} the normalized conformal map $\phi_1:\Ext C_1\to\D_*$ has the inverse
$$\phi_1^{-1}(w)=r_1w\frac {w-w_0+\frac {2} {r_1}\frac {|w_0|^2-1}{|w_0|^2}}{w-\bar{w}_0^{-1}},\qquad (w\in\D_*),$$
where $w_0$ is a real root to the equation
$$r_1^2 w_0^4-2r_1w_0^3-\alpha w_0^2-2r_1w_0+4=0.$$
The droplet is well defined with smooth boundary when $r_1>0$ is below a certain critical value $r_{1*}=r_{1*}(\alpha)$, at which the edge develops a cusp of type $(3,2)$. (The droplet corresponding to $r_{1*}$ is maximal and is understood in a local sense \cite{GrM}.)
\begin{center}
\begin{figure}
\includegraphics[width=0.6\textwidth]{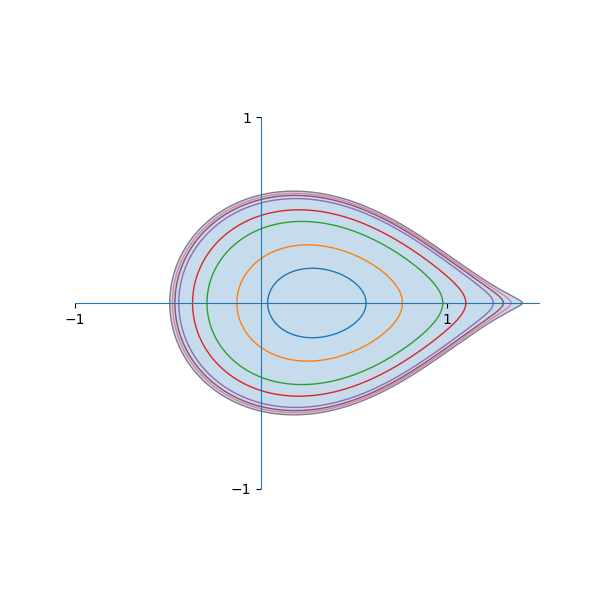}
\caption{Droplets corresponding to the potential \eqref{belem} with $\alpha = -0.5, \, r_1\in\{0.225, 0.375, 0.525, 0.6, 0.675, 0.712, 0.731\}$ and critical value $r_{1*}=0.75.$}

\label{fig3}
\end{figure}
\end{center}

Now fix any $r_1, r_2$ with $0<r_1<r_{1*}$ and $r_2>r_1$, and set
$$C_2:=\phi_1^{-1}(\frac {r_2}{r_1}\T).$$ We now define a family of outpost potentials having $C_2$ for an outpost. To accomplish this, we use the general device given in \cite[Section 1.3.1]{AC}. Namely, we fix an arbitrary constant $t_0>0$ and put
$$Q(z):=\begin{cases}Q_1(z),& z\in \overline{\calN}_1,\cr
\check{Q}_1(z)+t_0\dfrac {(|\phi_1(z)|^2-(r_2/r_1)^2)^2}{2|\phi_1(z)\phi_1'(z)|^2},& z\in\overline{\calN}_2,\cr
+\infty,& \mathrm{otherwise},\end{cases}$$
where $\calN_1$ is a small neighbourhood of $S$, $\calN_2$ a small neighbourhood of $C_2$, and $\overline{\calN}_1\cap\overline{\calN}_2=\emptyset.$ All conditions for an outpost potential are then satisfied with $h_1\equiv r_1$ and $c=\frac 1 2 \log \frac {t_0} {r_1}$. 

The outpost ensemble in Figure \ref{fig2} was generated in a similar way, but starting from an elliptic Ginibre potential $Q_1(z):=\frac 1 {r_1}(|z|^2+\re(\alpha z^2))$, where $\alpha$ is a complex parameter with $|\alpha|<1$.

\subsection{Comments and related work} \label{comrel} Outposts in the context of one-dimensional ensembles have been well studied, but tend to behave quite differently from the present ones. See \cite[Remark 4]{AC} for a comparison.

In \cite{ACC}, a large $n$ expansion for the free energy is given for radially symmetric potentials having a circular outpost in the unbounded component of the complement of the droplet. The outpost affects the constant term and can be expressed in terms of $q$-Pochhammer symbols. A generalization to the present kind of outpost potentials is conjectured in \cite{AC}.

Very recently, in \cite{N}, Kohei Noda obtained several new results in this direction. Among other things, he considers radially symmetric potentials such that the coincidence set accommodates an arbitrary number of circular outposts, and derives free energy asymptotics for such ensembles. This leads to the introduction of a new kind of multi-dimensional Heine distribution, which in a sense, accounts for interactions between different outposts.

We finally mention a few related problems, which we omit to discuss for reasons of length.

Firstly, it is natural to study diagonal asymptotics $K_n(z,z)$ when $z$ is in a vicinity of the outer edge $C_1$ of the droplet. This case differs from the above studied case, since bulk interactions enter the asymptotic picture. Indeed, if we fix $p\in C_1$
and scale about $p$ as in \eqref{scal}, then we have the leading order error-function asymptotic
\begin{align}\label{chod}K_n(z,z)=n\Delta Q(p)\cdot \frac {\erfc t} 2+o(n),\qquad (n\to\infty),\end{align}
which should be compared with Theorem \ref{mcor1}, in which $K_n(z,z)$ has order of magnitude $\sqrt{n}$. The asymptotics in \eqref{chod} has been verified at the edge of a wide variety of two-dimensional ensembles, cf.~the recent work \cite{CW} and references therein. A proof of \eqref{chod} in the present situation can either be based on the approximations in Lemma \ref{classlem} and Lemma \ref{usel} below, or by adapting the general method of proof in 
\cite{CW}. (We omit details.)

In addition to \eqref{chod}, subleading asymptotics has recently been found in a variety of settings \cite{A,ACC1,C} making it plausible that there should be an expansion of the form (still in the situation of \eqref{scal} with $p\in C_1$)
\begin{equation}K_n(z,z)=n\Delta Q(p)\frac {\erfc t} 2+\sqrt{n\Delta Q(p)}\cdot C(p,n;t)+\bigO(1),\qquad (n\to\infty),\end{equation}
where $C(p,n;t)$ oscillates in $n$, in a bounded way, as $n\to\infty$. It is an interesting problem (which we shall not consider here) to identify the coefficient $C(p,n;t)$.

Secondly, we recall that Szeg\H{o} type asymptotics for the correlation kernel is worked out for radially symmetric potentials with ring-shaped spectral gaps in \cite{ACC1}. In this case, the long-range correlations $K_n(p,q)$, where $p,q$ are distinct points along the edge of a spectral gap, ``oscillate'' in $n$. It is likely that these correlations should have universal generalizations, in a setting of spectral gap potentials as in \cite{AC}.
(In this connection, we note that a ``parallel'' but different kind of Szeg\H{o} type correlations, at the edge of an annular spectral gap with hard edge constraints,
is proved for the class of model Mittag-Leffler ensembles in \cite{ACC2}. Such droplets are not postcritical for Laplacian growth, but rather by imposing a hard wall.)

\subsection{Plan of this paper} In Section \ref{sec1} we derive various representations for the weighted Szeg\H{o} kernel $S_{1,2}(z.w)$, and we prove Lemma \ref{boll}.
In Section \ref{secprep} we discuss various relevant approximation formulas for the wavefunctions, i.e., the weighted orthogonal polynomials $e_{j,n}(z)$ with respect to the weight $e^{-nQ/2}$. We require two partially overlapping approximation formulas, one for the ``edge regime'', meaning indices $j$ with $n-C\sqrt{n\log n}\le j\le n-1$, and another for the ``bifurcation regime'', where $n-\log^2 n\le j\le n-1$. In Section \ref{thep}, we combine all these approximations and prove Theorem \ref{mth0}. 
In Section \ref{corc} we prove Theorem \ref{mth1} and Theorem \ref{mcor1}.
Finally, in Section \ref{berem} we prove Theorem \ref{massthm} and Corollary \ref{chuck}.

\section{Szeg\H{o} kernels} \label{sec1} In this section, we prove the formulas for the Szeg\H{o} kernel $S_{1,2}(z,w)$ in Lemma \ref{boll}, and we prove a few other representations that will come in handy.

To this end, we note that for $k=1,2$, the functions $$f_{j,k}(z):=\frac {\sqrt{r_k\phi_k'(z)}}{(r_k\phi_k(z))^j}e^{\frac 1 2 h_k(z)},\qquad (j=1,2,\ldots)$$
form an orthogonal basis for the Hilbert space $H_{1,2}$ (see \eqref{sqn}) with
$$\oint_{C_1}|f_{j,1}(z)|^2\frac {|dz|}{\sqrt{\Delta Q(z)}}=2\pi r_1^{1-2j},\qquad \oint_{C_1}|f_{j,2}(z)|^2\frac {|dz|}{\sqrt{\Delta Q(z)}}=e^{-c}\cdot 2\pi r_1^{2j-1}$$
and
$$\oint_{C_2}|f_{j,1}(z)|^2 \frac {|dz|}{\sqrt{\Delta Q(z)}}=e^{c}\cdot 2\pi r_2^{1-2j},\qquad \oint_{C_2}|f_{j,2}(z)|^2\frac {|dz|}{\sqrt{\Delta Q(z)}}=2\pi r_2^{1-2j}.$$

Now define
$$k_{j,1}=\|f_{j,1}\|_{1,2}^2=2\pi\cdot (r_1^{1-2j}+e^{c}r_2^{1-2j}),\qquad k_{j,2}=\|f_{j,2}\|_{1,2}^2=2\pi\cdot (r_1^{2j-1}e^{-c}+r_2^{1-2j}).$$
Then
\begin{align}S_{1,2}(z,w)&=\sum_{j=1}^\infty \frac {f_{j,1}(z)\overline{f_{j,1}(w)}}{k_{j,1}}\label{11form}\\
&=\frac 1 {2\pi}\sqrt{\phi_1'(z)}\overline{\sqrt{\phi_1'(w)}}e^{\frac 1 2(h_1(z)+\overline{h_1(w)})}\sum_{j=1}^\infty \frac 1 {(\phi_1(z)\overline{\phi_1(w)})^j}\frac {r_1^{1-2j}}{r_1^{1-2j}+e^{c}r_2^{1-2j}}.\nonumber
\end{align}
This proves Lemma \ref{boll}.

The formula \eqref{11form} is convenient in cases when both $z$ and $w$ are close to the curve $C_1$, since  $\phi_1(z)\overline{\phi_1(w)}$ is nearly unimodular.

For applications in which one or both the points $z,w$ are near $C_2$, we note the following alternative representations, which are obtained from \eqref{11form}
and the relationship $\phi_2=\frac {r_1}{r_2}\phi_1$
\begin{lem} \label{sluggo} The kernel $S_{1,2}(z,w)$ is given by
\begin{align}S_{1,2}(z,w)&
=\frac 1 {2\pi}\sqrt{\phi_2'(z)}\overline{\sqrt{\phi_2'(w)}}e^{\frac 1 2(h_2(z)+\overline{h_2(w)})}\sum_{j=1}^\infty \frac 1 {(\phi_2(z)\overline{\phi_2(w)})^j}\frac {r_2^{1-2j}}{e^{-c}r_1^{1-2j}+r_2^{1-2j}}\label{22form}\\
&=\frac 1 {2\pi} \sqrt{\phi_1'(z)}\overline{\sqrt{\phi_2'(w)}}e^{\frac 1 2(h_1(z)+\overline{h_2(w)})}\sum_{j=1}^\infty \frac 1 {(\phi_1(z)\overline{\phi_2(w)})^j}\frac {(r_1r_2)^{\frac 1 2-j}}{r_1^{1-2j}e^{-c/2}+e^{c/2}r_2^{1-2j}}.\label{12form}
\end{align}
\end{lem}

\section{Wavefunctions and their approximations} \label{secprep}

Recall the representation $K_n(z,w)=\sum_0^{n-1}e_{j,n}(z)\overline{e_{j,n}(w)}$, the wavefunction $e_{j,n}(z)$ being given by \eqref{wavefunction}.

We are particularly interested in asymptotics when both $z$ and $w$ are in $\calB_1\cup \Ext C_1$. 
In this case, terms in \eqref{ck} with $j\le n-C\sqrt{n\log n}$ give a negligible contribution, and the sum can be approximated by just summing over $j$:s with $j=n+\bigO(\sqrt{n\log n})$.

\begin{lem} \label{zoff} If $j\le n-C\sqrt{n\log n}$ where the constant $C$ is large enough, then for each $N>0$ there is a constant $K_N$ such that 
$$|e_{j,n}(z)|\le K_Nn^{-N},\qquad (z\in\calB_1\cup \Ext C_1).$$
\end{lem}

The proof is standard, e.g.~see \cite[Lemma 3.6]{AC}.

\smallskip

When the points $z,w\in \calB_1\cup \Ext C_1$, the essential contribution to the sum \eqref{ck} comes from indices $j$
 with $n-M\sqrt{n\log n}\le j\le n-1$. These indices constitute the \textit{edge regime}. 
 
 As observed in \cite{AC}, in the present setting a further \textit{bifurcation regime} is prevalent for indices $j$ with $ n-\log^2 n\le j\le n-1$. 
 
 To account for all possibilities, we now state two partially overlapping approximation formulas for the wavefunctions $e_{j,n}(z)$, to be used in tandem in what follows.

We first recall the more well known approximation formula used for connected coincidence sets going back to \cite{HW}. 

For all $j$ with $n-M\sqrt{n\log n}\le j\le n-1$ we put $\tau=j/n$ and consider the potential $Q_\tau=Q/\tau$. Let $S_\tau=S[Q/\tau]$ be the $\tau$-droplet. These droplets form an increasing chain: $\tau<\tau'$ implies $S_{\tau}\subset S_{\tau'}$.

Also (at least for $n$ large enough) $\d S_\tau$ is a real-analytic Jordan curve, see e.g.~\cite{AC1,GrM} and references therein. Let $U_\tau:=\hat{\C}\setminus S_\tau$ be the exterior domain and
$$\phi_\tau:U_\tau\to\D_*$$
the exterior conformal map with $\phi_\tau(\infty)=\infty$ and $\phi_\tau'(\infty)>0$.

We also require the two analytic functions $q_\tau(z)$ and $h_\tau(z)$ on $U_\tau$ which solve the boundary value problems
\begin{align}\re q_\tau=Q,\qquad \re h_\tau=\frac 1 2\log\Delta Q,\qquad \text{along}\quad \d S_\tau,\end{align}
normalized by mapping $\infty\mapsto \infty$ and being real at $\infty$.

To state our next result, we also introduce the ``obstacle function'' $\check{Q}_\tau(z)$. By definition $\check{Q}_\tau(z)$ is the pointwise supremum over $s(z)$ where $s(w)$ is a subharmonic function on $\C$ such that $s(w)\le 2\tau\log|w|+\bigO(1)$ as $w\to\infty$. The coincidence set $S_\tau^*:=\{Q_\tau=\check{Q}_\tau$ satisfies $S_\tau^*=S_\tau$ when $\tau<1$ is close enough to $1$; also $\check{Q}_\tau(z)=2\tau\log|z|+\bigO(1)$ as $z\to\infty$.

We now define an approximation (or ``weighted quasi-polynomial'') $E_{j,n}(z)$ by
\begin{align}\label{qp1}E_{j,n}(z):=\bigg(\frac n {2\pi}\bigg)^{\frac 1 4}\sqrt{\phi_\tau'(z)}\cdot\phi_\tau(z)^j\cdot e^{-\frac n 2(Q-q_\tau)(z)}\cdot e^{\frac 1 2 h_\tau(z)}.
\end{align}

The exterior approximation ``$e_{j,n}\approx E_{j,n}$'' works well as long as $j$ does not enter the bifurcation regime.

\begin{lem}\label{classlem} Suppose that $j$ satisfies $n-C\sqrt{n\log n}\le j\le n-\log^2 n$. 
Then for all $z\in\calB_1\cup \Ext C_1$ we have the estimate
\begin{equation}|e_{j,n}(z)-E_{j,n}(z)|\le K\sqrt{\log n}\cdot e^{-\frac n 2 (Q-\check{Q}_\tau)(z)},\end{equation}
where $K$ is a large enough constant (depending only on $Q$).
\end{lem}

\begin{proof}
 For $j\le n-\log^2 n$ the effect of the outpost is negligible by \cite[Lemma 3.7]{AC}, i.e., if $\calE$ is an arbitrarily small neighbourhood of the droplet $S$, then $e_{j,n}(z)$ as well as $E_{j,n}(z)$ are uniformly $\bigO(n^{-N})$ on $\C\setminus\calE$ for any given $N>0$. When applying the standard approximation scheme in \cite{HW,AC1} (using ``H\"{o}rmander estimates'') it is easy to see that the approximation with $E_{j,n}(z)$ works virtually unaltered. We omit details. \footnote{It is equally easy to see that the approximation procedure works for the higher order approximants in \cite{HW}. However, we shall here merely be concerned with the first-order approximation by $E_{j,n}(z)$.}
\end{proof}

For $j$ in the bifurcation regime, an exterior approximation
formula for the wavefunction $e_{j,n}(z)$ is deduced in \cite[Section 3.4]{AC}. We now recall this formula.

Recall that $\phi_k:\Ext C_k\to\D_*$, $k=1,2$ are the exterior conformal mappings.

For $j$ close to $n$ we consider the holomorphic function in $\calB_1\cup \Ext C_1$
$$\Phi_{j,n}(z)=\frac {r_1^{j+1/2}}{e^{\frac n 2 q_1(\infty)+\frac 1 2 h_1(\infty)}}\sqrt{\phi_1'(z)}\phi_1(z)^je^{\frac n 2 q_1(z)}e^{\frac 1 2 h_1(z)}.$$
Here $q_1,h_1$ are holomorphic functions in $\Ext C_1$ with
$$\re q_1=Q,\qquad \re h_1=\frac 1 2\log\Delta Q\qquad \mathrm{along}\quad C_1,$$
fixed by the condition that $q_1$ and $h_1$ are real at infinity.

As shown in \cite[Section 2]{AC}, we have for $z\in\Ext C_2$
$$\Phi_{j,n}(z)=\frac {r_2^{j+1/2}}{e^{\frac 1 2 nq_2(\infty)+\frac 1 2 h_2(\infty)}}\sqrt{\phi_2'(z)}\phi_2(z)^je^{\frac 1 2 nq_2(z)}e^{\frac 1 2 h_2(z)}.$$
Here $q_2,h_2$ are the holomorphic on $\Ext C_2$ solving the boundary problems
$$\re q_2=Q,\qquad \re h_2=\frac 1 2\log \Delta Q,\qquad \mathrm{along}\quad C_2,$$
and real at infinity.

It is easy to see (cf.~\cite[Section 1.3]{AC}) that
\begin{equation}\label{brel}h_2=h_1+c\end{equation}
where $c$ is the constant in \eqref{constc}.

Let us also set 
$$c_{j,n}:=\sqrt {\frac {2\pi} {n}}(r_1^{2j+1}e^{-nq_1(\infty)-h_1(\infty)}+r_2^{2j+1}e^{-nq_2(\infty)-h_2(\infty)})$$
and 
\begin{equation}\label{mp}F_{j,n}(z):=c_{j,n}^{-\frac 12}\Phi_{j,n}(z)e^{-\frac 12 nQ(z)}.\end{equation}

\begin{lem} \label{usel} Suppose that $n-\log^2 n\le j\le n-1$. Then for all $z\in \calB_1\cup \Ext C_1$ we have
the estimate
$$|e_{j,n}(z)-F_{j,n}(z)|\le K\sqrt{\log n}\cdot e^{-\frac 1 2 n(Q-\check{Q})(z)},$$
where $K$ is a large enough constant depending only on $Q$.
\end{lem}

\begin{proof} This is precisely
\cite[Theorem 3.5]{AC}. 
\end{proof}

\section{Approximation of the correlation kernel} \label{thep} 
We now prove Theorem \ref{mth0}. The idea is inspired by the paper \cite{ACC}.

For $z,w\in \calB_1\cup\Ext C_1$ we consider the approximation of $K_n(z,w)$ by the sum $\Sigma_1^{(n)}(z,w)+\Sigma_{1,2}^{(n)}(z,w)$
where
\begin{align}\Sigma_1^{(n)}(z,w):=\sum_{j=n-C\sqrt{n\log n}}^{n-1}E_{j,n}(z)\overline{E_{j,n}(w)}\end{align} 
and
\begin{align}\label{second}\Sigma_{1,2}^{(n)}(z,w):=\sum_{j=n-\log^2 n}^{n-1}(F_{j,n}(z)\overline{F_{j,n}(w)}-E_{j,n}(z)\overline{E_{j,n}(w)}).\end{align}

The first term $\Sigma_1^{(n)}$ requires no new analysis; it is easily handled using the method of proof of \cite[Theorem 1.3]{AC1}.

\begin{lem} Let $\beta$ be any number with $0<\beta<\frac 1 4$. Then for all $z,w\in\calB_1\cup \Ext C_1$ such that $|\phi_1(z)\overline{\phi_1(w)}-1|\ge \eta$ we have the estimate
\begin{align*}\Sigma_1^{(n)}(z,w)=\sqrt{2\pi n}\, &(u(z)\overline{u(w)})^ne^{-\frac n 2 (Q(z)+Q(w))} 
\cdot S_1(z,w)\cdot (1+\bigO(n^{-\beta}))
\end{align*}
as $n\to\infty$, where $S_1(z,w)$ is the Szeg\H{o} kernel \eqref{s1def}.
\end{lem}

\begin{proof} The proof, using summation by parts, in \cite[Section 4]{AC1}, works unaltered in the present situation. 
\end{proof}

We turn to the term $\Sigma_{1,2}^{(n)}(z,w)$. We have the following lemma.

\begin{lem} \label{bok} Suppose that $z,w\in\calB_1\cup\Ext C_1$.
Then as $n\to\infty$,
\begin{align*}\Sigma_{1,2}^{(n)}(z,w)&=\sqrt{\frac  n {2\pi}}\sqrt{\phi_1'(z)}\overline{\sqrt{\phi_1'(w)}}(u(z)\overline{u(w)})^ne^{-\frac n 2 (Q(z)+Q(w))}
e^{\frac 1 2 (h_1(z)+\overline{h_1(w)})}\\
&\times  \sum_{j=1}^{\infty} \frac 1 {(\phi_1(z)\overline{\phi_1(w)})^j}\frac {e^{-c}r_2^{1-2j}}{r_1^{1-2j}+e^{-c}r_2^{1-2j}}\cdot \bigg(1+\bigO\bigg(\frac {\log^4 n} n\bigg)\bigg). 
\end{align*}
\end{lem}

\begin{proof}

By \cite[Lemma 2.1]{AC} we have
$$e^{q_1(\infty)-q_2(\infty)}=\bigg(\frac {r_1}{r_2}\bigg)^2.$$
Using also $h_2(\infty)=h_1(\infty)+c$, we now write
\begin{align}\Sigma_{1,2}^{(n)}(z,w)&=\sqrt{\frac n {2\pi}}\sqrt{\phi_1'(z)}\overline{\sqrt{\phi_1'(w)}}\cdot(\phi_1(z)\overline{\phi_1(w)})^n\cdot e^{\frac n 2(q_1(z)-Q(z))}e^{\frac n 2(\overline{q_1(w)}-Q(w))}e^{\frac 1 2(h_1(z)+\overline{h_1(w)})}\nonumber\\
&\quad \times \sum_{j=n-\log^2 n}^{n-1}(\phi_1(z)\overline{\phi_1(w)})^{j-n}\cdot \bigg[\frac {r_1^{2j+1}}{r_1^{2j+1}+r_2^{2j+1}(\frac {r_1}{r_2})^{2n}e^{-c}}\label{hugo} \\
&\qquad\qquad -\frac {\sqrt{\phi'_\tau(z)}\overline{\sqrt{\phi'_\tau(w)}}}
{\sqrt{\phi_1'(z)}\overline{\sqrt{\phi_1'(w)}}}\bigg(\frac {\phi_\tau(z)\overline{\phi_\tau(w)}}{\phi_1(z)\overline{\phi_1(w)}}\bigg)^j\frac {e^{\frac n 2 (q_\tau(z)+\overline{q_\tau(w)})}}{e^{\frac n 2(q_1(z)+\overline{q_1(w)})}}\frac {e^{\frac 1 2(h_\tau(z)+\overline{h_\tau(w)}}}{e^{\frac 1 2(h_1(z)+\overline{h_1(w)})}}\bigg].\nonumber
\end{align}

We shall next prove that 
\begin{align}\frac {\sqrt{\phi'_\tau(z)}\overline{\sqrt{\phi'_\tau(w)}}}
{\sqrt{\phi_1'(z)}\overline{\sqrt{\phi_1'(w)}}}\bigg(\frac {\phi_\tau(z)\overline{\phi_\tau(w)}}{\phi_1(z)\overline{\phi_1(w)}}\bigg)^j\frac {e^{\frac n 2 (q_\tau(z)+\overline{q_\tau(w)})}}{e^{\frac n 2(q_1(z)+\overline{q_1(w)})}}\frac {e^{\frac 1 2(h_\tau(z)+\overline{h_\tau(w)}}}{e^{\frac 1 2(h_1(z)+\overline{h_1(w)})}}=1+\bigO\bigg(\frac {\log^4 n} n\bigg).\label{tos}
\end{align}

This is a little technical, but fortunately the main estimates have already been carried out in \cite{AC1}.

Indeed, to show \eqref{tos}, we write $V_\tau(z)$ for the harmonic continuation of $\check{Q}_\tau(z)$ from $\C\setminus S_\tau$ inwards across $\d S_\tau$. It is well-known and easy to check that
$$|\phi_\tau(z)|^je^{\frac 2 n\re q_\tau(z)}=e^{\frac n 2 V_\tau(z)}.$$

Proceeding as in the proof of \cite[Lemma 4.4]{AC1} we define 
$$P_\tau(z):=\tau\log\left[\frac {\phi_\tau(z)}{\phi_1(z)}e^{\frac 1 {2\tau}(q_\tau-q_1)(z)}\right],$$
with the branch of the logarithm such that $P_\tau(\infty)$ is real. Note that $P_1\equiv 0$.

The proof of \cite[Lemma 4.4]{AC1} shows that there is a strictly positive constant $a$ such that
$$\re P_\tau(z)\le - a(1-\tau)^2$$
which leads to
$$\bigg|\bigg(\frac {\phi_\tau(z)\overline{\phi_\tau(w)}}{\phi_1(z)\overline{\phi_1(w)}}\bigg)^j\frac {e^{\frac n 2 (q_\tau(z)+\overline{q_\tau(w)})}}{e^{\frac n 2(q_1(z)+\overline{q_1(w)})}}\bigg|=e^{\frac n 2(\re P_\tau(z)+\re P_\tau(w))}\le e^{-a\frac {(n-j)^2} n}.$$
For $1\le n-j\le \log^2 n$ the right hand side is $1+\bigO(n^{-1}\log^4 n)$.

Inserting this asymptotic in \eqref{hugo} and rearranging we find
\begin{align*}\Sigma_{1,2}^{(n)}(z,w)&=\sqrt{\frac n {2\pi}}\sqrt{\phi_1'(z)}\overline{\sqrt{\phi_1'(w)}}(u(z)\overline{u(w)})^ne^{-\frac n 2(Q(z)+Q(w))}e^{\frac 1 2(h_1(z)+\overline{h_1(w)})}\\
&\times \sum_{j=1}^{\log^2 n} \frac 1 {(\phi_1(z)\overline{\phi_1(w)})^j}\cdot \bigg[\frac {e^{-c}r_2^{1-2j}}{r_1^{1-2j}+r^{-c}r_2^{1-2j}}+\bigO(n^{-1}\log^4 n)\bigg].
\end{align*} The lemma now follows from the representation \eqref{extend} of $S_{1,2}(z,w)$.
\end{proof}

The last two lemmas conclude our proof of Theorem \ref{mth0}. $\qed$

\section{Correlations between points in $\calB_1\cup\calB_2$} \label{corc}

We now prove Theorem \ref{mth1} and Theorem \ref{mcor1}. We start with the following lemma.

\begin{lem} \label{slugger} Let $V(z)$ be the harmonic continuation of $\check{Q}|_{\Ext C_1}$ inwards across $C_1$. For $p\in C_1\cup C_2$ set
\begin{equation}\label{zo0}z=p+\frac s {\sqrt{2n\Delta Q(p)}}\nu(p)\end{equation}
where $\nu(p)$ is the exterior unit normal. We then have the Taylor expansion
$$(Q-V)(z)=s^2+\bigO\big( |s|^3/\sqrt{n}\big)$$
where the implied constant is uniform for all $|s|\le M\sqrt{\log n}$.
\end{lem}

\begin{proof} Let $\d_\ttn$ denote differentiation in the exterior normal direction to $C_k$, where $p\in C_k$.
Using that $Q-V$ is real-analytic near $C_k$, Taylor's formula gives, as $\ell\to 0$, $\ell\in\R$
\begin{align*}(Q-V)(p+\ell\nu(p))&=\frac 1 2 \d_\ttn^2 (Q-V)(p)\cdot \ell^2+\bigO(\ell^3)\\
&=2\Delta (Q-V)(p)\cdot\ell^2+\bigO(\ell^3)\\
&=2\Delta Q(p)\cdot\ell^2+\bigO(\ell^3),\end{align*}
where we used that $V$ is harmonic.
\end{proof}

Assume next that both $p$ and $q$ are in $C_1\cup C_2$ and that $|\phi_1(p)\overline{\phi_1(q)}-1|\ge \eta>0$.

 Also fix $s,t$ with $|s|,|t|\le M\sqrt{\log n}$ and let $z$ and $w$ be as in \eqref{zoom}.

Using Lemma \ref{slugger}, we see that 
$$(u(z)\overline{u(w)})^ne^{-\frac n 2(Q(z)+Q(w))}=e^{-\frac 1 2(s^2+t^2)}\cdot \exp(\bigO((\log n)^{\frac 32}n^{-\frac 12})).$$

Inserting this in Theorem \ref{mth0}, we finish our proof of Theorem \ref{mth1}.

Next use Theorem \ref{mth1} and Lemma \ref{sluggo} to conclude that for $z$ as in \eqref{zo0} and $p\in C_2$,
\begin{align*}K_n(z,z)&=\sqrt{2\pi n}\,e^{-s^2}S_{1,2}(p,p)\cdot (1+\bigO(n^{-\beta}))\\
&=\sqrt{2\pi n}\, e^{-s^2}|\phi_2'(p)|\sqrt{\Delta Q(p)}\frac 1 {2\pi}\sum_{j=1}^\infty \frac {r_2^{1-2j}}{r_1^{1-2j}e^{-c}+r_2^{1-2j}}\cdot (1+\bigO(n^{-\beta})).
\end{align*} 
By \eqref{expheine} we have
$$\sqrt{\Delta Q(p)}\sum_{j=1}^\infty\frac {r_2^{1-2j}}{r_1^{1-2j}e^{-c}+r_2^{1-2j}}=\mu,$$
finishing the proof of Theorem \ref{mcor1}. q.e.d.

\section{Berezin measures and the expectation formula} \label{berem}
In this section we prove Theorem \ref{massthm} and Corollary \ref{chuck}.
We begin with Theorem \ref{massthm}.

Recall that the Berezin measure rooted at a given point $z\in \Ext C_1$ is the probability measure
$$d\mu_{n,z}(w)=\frac {|K_n(z,w)|^2}{K_n(z,z)}\, dA(w).$$

Let $\calB_1$ and $\calB_2$ be the belts about $C_1$ and $C_2$ respectively, see \eqref{2cho}. We start with the following lemma, which says that $\mu_{n,z}$ is essentially supported on the union of the belts.

\begin{lem} \label{stuck} For fixed $z\in\Ext C_1$ and any $N>0$ we have the convergence
\begin{align}\label{supp}\mu_{n,z}(\C\setminus (\calB_1\cup\calB_2))=\bigO(n^{-N}),\qquad (n\to\infty).
\end{align}
\end{lem}

\begin{proof}
By Corollary 
\ref{qm}, we infer that  (taking the constant $M$ large enough) $\mu_{n,z}((\Ext C_1)\setminus(\calB_1\cup\calB_2))=\bigO(n^{-N})$.
It remains to consider the measure of the domain $I_n=(\Int C_1)\setminus \calB_1$. For this, it is convenient to invoke the off-diagonal estimate in \cite[Theorem 8.1]{CPAM}, which implies that we can, by choosing $M$ large enough, ensure that 
\begin{align}\label{comb1}|K_n(z,w)|^2\le C_Nn^{-N}e^{-n(Q-\check{Q})(z)},\qquad (w\in I_N).\end{align}
On the other hand, by Theorem \ref{mth0} we have that 
\begin{align}\label{comb2}K_n(z,z)=\sqrt{2\pi n}\cdot e^{-n(Q-\check{Q})(z)}\cdot S_{1,2}(z,z)\cdot (1+\bigO(n^{-\beta})).\end{align}
Combining \eqref{comb1} and \eqref{comb2} we see that $\mu_{n,z}(I_n)=\bigO(n^{\frac 1 2-N})$.
\end{proof}

Now parameterize $\calB_2$ by $(q,t)$,
\begin{equation}\label{cov}w=w(q,t):=q+\frac t{\sqrt{2n\Delta Q(q)}}\cdot \nu(q),\qquad (q\in\d S,\, |t|\le M\sqrt{\log n}).\end{equation}

A computation of the Jacobian (see \cite[Lemma 3.1]{A}) shows that
$$dA(w)=\frac 1 \pi \frac 1 {\sqrt{2n\Delta Q(q)}}\cdot \bigg(1+\frac {t\kappa(q)}{\sqrt{2n\Delta Q(q)}}\bigg)\, dt\, ds(q),\qquad (w=w(q,t)),$$
where $\kappa(q)=\d_\tts \arg \nu(q)$ is the signed curvature of $\d S$.

A similar relation holds for $w\in\calB_1$.

Let $f(w)$ an arbitrary function in $H_{1,2}$ which extends continuously to the boundary $\d U$. Extend $f$ to a continuous function on $\hat{\C}$ in some way. 

Using Theorem \ref{mth1} and Lemma \ref{stuck} we have
\begin{align*}\int_\C &f\,d\mu_{n,z}=\sum_{k=1}^2\int_{\calB_k}f\, d\mu_{n,z}+\bigO(n^{-N})\|f\|_\infty\\
&=\sum_{k=1}^2\sqrt{2\pi n}\frac 1 \pi \int_{-M\sqrt{\log n}}^{M\sqrt{\log n}}f(w)\, e^{-t^2}\, dt\cdot (1+\bigO(n^{-\beta}))\\
&\qquad\times\oint_{C_k}\sum_{k=1,2}\sqrt{\Delta Q(q)}\frac {|S_{1,2}(z,q)|^2}{S_{1,2}(z,z)}\frac 1 {\sqrt{2n\Delta Q(q)}}\cdot \bigg(1+\frac {t\kappa(q)}{\sqrt{2n\Delta Q(q)}}\bigg)\, ds(q)\\
&\qquad\qquad \qquad +\|f\|_\infty\cdot \bigO(n^{-N})\\
&\qquad=\sum_{k=1,2}\oint_{C_k}f(q)\frac {|S_{1,2}(z,q)|^2}{S_{1,2}(z,z)}\, ds(q)+\bigO(n^{-N})\|f\|_\infty.
\end{align*}

Denote
$$b_z^{(k)}(f):=\oint_{C_k} f(q)\frac {|S_{1,2}(z,q)|^2}{S_{1,2}(z,z)}\, ds(q),\qquad (k=1,2).$$

Since $q\mapsto f(q)S_{1,2}(z,q)$ is in $H_{1,2}$, the reproducing property of $S_{1,2}$ gives
$$b_z^{(1)}(f)+b_z^{(2)}(f)=f(z).$$

Now fix $z\in U$ and write $r=|\phi_2(z)|$, so $r>r_1/r_2$. We can without loss assume that $\phi_2(z)$ is real. Then
\begin{align*}S_{1,2}(z,w)&=\frac 1 {2\pi}\sqrt{\phi_2'(z)}\overline{\sqrt{\phi_2'(w)}}e^{\frac 1 2(h_2(z)+\overline{h_2(w)})}\sum_{j=1}^\infty \frac
1 {(\overline{\phi_2(w)})^j}\frac {r^{-j}r_2^{1-2j}}
{r_1^{1-2j}e^{-c}+r_2^{1-2j}}
\end{align*}
whence, on setting $\phi_2(q)=e^{i\theta}$, $ds(q)=|\phi_2'(q)|^{-1}\, d\theta$,
\begin{align*}\oint_{C_2} |S_{1,2}(z,q)|^2\, ds(q)&=\frac{|\phi_2'(z)|}{4\pi^2}\int_0^{2\pi}\bigg|\sum_{j=1}^\infty \frac {r^{-j}r_2^{1-2j}}
{r_1^{1-2j}e^{-c}+r_2^{1-2j}} e^{ij\theta}\bigg|^{\,2}\, d\theta\\
&=\frac {|\phi_2'(z)|}{2\pi}\sum_{j=1}^\infty\bigg(\frac {r^{-j}r_2^{1-2j}}
{r_1^{1-2j}e^{-c}+r_2^{1-2j}}\bigg)^{\,2},
\end{align*}
where the last equality uses Parseval's theorem.

Since
$$S_{1,2}(z,z)=\frac {|\phi_2'(z)|} {2\pi}\sum_{j=1}^\infty r^{-2j}\frac {r_2^{1-2j}}{r_1^{1-2j}e^{-c}+r_2^{1-2j}}$$
we obtain the result that
\begin{align*}b_z^{(2)}(1)=\frac {\sum_{j=1}^\infty r^{-2j}\bigg(\frac { r_2^{1-2j}}
{r_1^{1-2j}e^{-c}+r_2^{1-2j}}\bigg)^{\,2}}{\sum_{j=1}^\infty r^{-2j}\frac {r_2^{1-2j}}{r_1^{1-2j}e^{-c}+r_2^{1-2j}}}.
\end{align*}
This finishes the proof of Theorem \ref{massthm}. q.e.d.

\smallskip

We finally prove Corollary \ref{chuck}.
Using Theorem \ref{mcor1} and the change of variables \eqref{cov}, we obtain
\begin{align*}\int_{\calB_2}K_n(z,z)\, dA(z)&=\frac 1 \pi\oint_{C_2}\sqrt{\frac {n\Delta Q(p)}{2\pi}}\cdot \mu\cdot|\phi_2'(p)|\,\frac {|dp|}{\sqrt{2n\Delta Q(p)}}\\
&\qquad \qquad \times\int_{-M\sqrt{\log n}}^{M\sqrt{\log n}}e^{-t^2}\, dt\cdot(1+\bigO(n^{-\beta}))\\
&=\frac \mu {2\pi}\int_\T|dz|\cdot (1+\bigO(n^{-\beta}))=\mu\cdot (1+\bigO(n^{-\beta})).
\end{align*}
The proof is complete. q.e.d.

\subsection*{Acknowledgment} The authors are grateful to Joakim Cronvall for discussions and much appreciated help.

\end{document}